\documentclass[onefignum,onetabnum]{siamart190516}
\usepackage{amsmath}
\usepackage{amssymb}
\usepackage{enumerate}
\usepackage{mathrsfs}
\usepackage{cases}
\usepackage{color}
\usepackage{tikz}
\usepackage{mathtools}
\usepackage{makecell}
\usepackage{multirow}
\usepackage{siunitx}
\usepackage{graphicx}
\definecolor{darkblue}{rgb}{0.0, 0.0, 0.55}
\definecolor{bordeaux}{rgb}{0.34, 0.01, 0.1}
\newsiamthm{example}{Example}
\newtheorem{conjecture}[theorem]{Conjecture}

\def\R{{\mathbb{R}}}
\def\C{{\mathbb{C}}}
\def\N{{\mathbb{N}}}

\def\z{{\mathbf{z}}}
\def\e{{\mathbf{e}}}
\def\y{{\mathbf{y}}}

\def\ba{{\mathbf{a}}}

\def\a{{\boldsymbol{\alpha}}}

\def\b{{\boldsymbol{\beta}}}
\def\g{{\boldsymbol{\gamma}}}
\def\bv{{\boldsymbol{v}}}

\def\RR{{\mathscr{R}}}

\def\M{{\mathbf{M}}}

\def\i{\hbox{\bf{i}}}

\def\rank{\hbox{\rm{rank}}}
\def\Span{\hbox{\rm{span}}}

\usepackage{lipsum}
\usepackage{amsfonts}
\usepackage{graphicx}
\usepackage{epstopdf}
\usepackage{algorithmic}
\ifpdf
  \DeclareGraphicsExtensions{.eps,.pdf,.png,.jpg}
\else
  \DeclareGraphicsExtensions{.eps}
\fi

\newsiamremark{remark}{Remark}
\newsiamremark{hypothesis}{Hypothesis}
\crefname{hypothesis}{Hypothesis}{Hypotheses}
\newsiamthm{claim}{Claim}

\headers{A real moment-HSOS hierarchy for complex polynomial optimization}{Jie Wang and Victor Magron}

\title{A real moment-HSOS hierarchy for complex polynomial optimization with real coefficients\thanks{Submitted to the editors DATE.
\funding{This work was funded by NSFC-12201618, NSFC-12171324, by the PEPS2 FastOPF funded by RTE and the French Agency for mathematics in interaction with industry and society (AMIES), the EPOQCS grant funded by the LabEx CIMI (ANR-11-LABX-0040), the European Union’s Horizon 2020 research and innovation programme under the Marie Sk{\l}odowska-Curie Actions, grant agreement 813211 (POEMA), by the AI Interdisciplinary Institute ANITI funding, through the French ``Investing for the Future PIA3'' program under the Grant agreement n${}^\circ$ ANR-19-PI3A-0004 as well as by the National Research Foundation, Prime Minister’s Office, Singapore under its Campus for Research Excellence and Technological Enterprise (CREATE) programme.}
}}

\author{Jie Wang\thanks{Academy of Mathematics and Systems Science, Chinese Academy of Sciences, Beijing, China
  (\email{wangjie212@amss.ac.cn}, \url{https://wangjie212.github.io/jiewang/})}
\and Victor Magron\thanks{LAAS CNRS \& IMT, Toulouse, France 
   (\email{vmagron@laas.fr})}}

\usepackage{amsopn}
\DeclareMathOperator{\diag}{diag}

\ifpdf
\hypersetup{
  pdftitle={An Example Article},
  pdfauthor={D. Doe, P. T. Frank, and J. E. Smith}
}
\fi

\begin{document}

\maketitle

\begin{abstract}
This paper proposes a real moment-HSOS hierarchy for complex polynomial optimization problems with real coefficients. We show that this hierarchy provides the same sequence of lower bounds as the complex analogue, yet is much cheaper to solve. In addition, we prove that global optimality is achieved when the ranks of the moment matrix and certain submatrix equal two in case that a sphere constraint is present, and as a consequence, the complex polynomial optimization problem has either two real optimal solutions or
a pair of conjugate optimal solutions. A simple procedure for extracting a pair of conjugate optimal solutions is given in the latter case. Various numerical examples are presented to demonstrate the efficiency of this new hierarchy, and an application to polyphase code design is also provided.
\end{abstract}



\section{Introduction}
A complex polynomial optimization problem (CPOP) takes the following general form:
\begin{equation}\label{cpop}
f_{\min}\coloneqq\begin{cases}
\inf\limits_{\z\in\C^n} &f(\z,\overline{\z})\\
\,\,\,\rm{s.t.}&g_i(\z,\overline{\z})\ge0,\quad i=1,\ldots,m,\\
&h_i(\z,\overline{\z})=0,\quad i=1,\ldots,l,
\end{cases} \tag{CPOP}
\end{equation}
where $\overline{\z}\coloneqq(\overline{z}_1,\ldots,\overline{z}_n)$ stands for the conjugate of complex variables $\z\coloneqq(z_1,\ldots,z_n)$, and $f,g_1,\ldots,g_m,h_1,\ldots,h_l$ are polynomials in $\z,\overline{\z}$ that are invariant under conjugation (hence take real values).
\eqref{cpop} arises naturally from diverse areas, such as imaging science \cite{fogel2016phase}, signal processing \cite{aittomaki2009,aubry2013ambiguity,dumitrescu2007positive,mariere2003}, automatic control \cite{toker1998complexity}, quantum mechanics \cite{hilling2010}, power systems \cite{bienstock2020}.
Note that by introducing real variables for the real and imaginary parts of each complex variable, respectively, \eqref{cpop} can be converted into a polynomial optimization problem (POP) involving only real variables.

Nowadays a popular scheme to globally solve POPs is the so-called moment-sum-of-squares (moment-SOS) hierarchy (also known as Lasserre's hierarchy) \cite{Las01} which consists of a sequence of increasingly tight semidefinite relaxations. In a similar spirit, the complex moment-Hermitian-sum-of-squares (moment-HSOS) hierarchy involving positive semidefinite Hermitian matrices has also been established to handle CPOPs \cite{josz2018lasserre}. We remark that to solve a CPOP, one can either apply the real moment-SOS hierarchy after converting it into an equivalent real POP or directly apply the complex moment-HSOS hierarchy. It is known that the complex moment-HSOS hierarchy may produce weaker bounds than the real moment-SOS hierarchy at the same relaxation order, which is, however, still of interest because of its lower computational complexity \cite{josz2018lasserre,wang2022exploiting}.

A common drawback of the real moment-SOS hierarchy and the complex moment-HSOS hierarchy is that the size of the related semidefinite relaxations grows rapidly with both the number of polynomial variables and the relaxation order, which makes them hardly applicable to large-scale problems.
In order to improve scalability, it is crucial to exploit structures encoded in problem data to develop structured hierarchies. Such structures include correlative sparsity \cite{josz2018lasserre,waki}, term sparsity \cite{tssos2,tssos1,tssos3,wang2022exploiting}, symmetry \cite{riener2013exploiting}.
Structured hierarchies have been successfully employed to tackle lots of practical applications, for instance, estimating roundoff errors of computer arithmetic \cite{toms17}, bounding joint spectral radius \cite{wang2020sparsejsr}, computing optimal power flow \cite{wang2022certifying}, quantum optimal control \cite{bondar2022quantum}, neural network verification \cite{newton2022sparse}, just to name a few. 

In this paper, we are concerned with a particular subclass of \eqref{cpop}, where the polynomials $f,g_1,\ldots,g_m,h_1,\ldots,h_l$ all have real coefficients, or equivalently, the polynomials $f,g_1,\ldots,g_m,h_1,\ldots,h_l$ are invariant under the conjugate transformation: $\z\mapsto\overline{\z}$. Many complex polynomial optimization problems involving norms (e.g., the problem of polyphase code design) belong to this class (see Sections \ref{expe} and \ref{app}).
For this type of \eqref{cpop}, we propose a symmetry-adapted moment-HSOS hierarchy that involves positive semidefinite {\bf real} matrices. We show that, in comparison to the complex moment-HSOS hierarchy, this new hierarchy (1) has {\bf lower complexity} (see Table \ref{tab:comp}), and (2) provides the {\bf same bound} at each relaxation order (Theorem \ref{thm3}). Essentially, this result is an extension of Riener et al.'s work \cite{riener2013exploiting} to complex polynomial optimization with conjugate symmetry. Furthermore, as another contribution of this paper, we prove that global optimality of the new hierarchy is attained when the ranks of the moment matrix and certain submatrix equal two in the presence of a sphere constraint. Note that here we do not require the ``joint hyponormality'' condition to declare global optimality (cf. \cite[Proposition 4.1]{josz2018lasserre}). The rank-two condition also implies that the related complex polynomial optimization problem has either two real optimal solutions or a pair of conjugate optimal solutions. In the latter case, we provide a simple procedure for extracting a pair of conjugate global minimizers from the moment matrix of order one.

The efficiency of the new hierarchy is demonstrated on randomly generated CPOP instances as well as two CPOPs (respectively related to Smale’s Mean Value conjecture and the Mordell inequality conjecture) from the literature. Finally, as an application, we apply the new hierarchy to the important problem of polyphase code design arising from signal processing. 

\section{Notation and preliminaries}\label{preliminaries}
Let $\N$ be the set of nonnegative integers.
For $n\in\N\setminus\{0\}$, let $[n]\coloneqq\{1,2,\ldots,n\}$. For $\a=(\alpha_i)\in\N^n$, let $|\a|\coloneqq\sum_{i=1}^n\alpha_i$. For $r\in\N$, let $\N^n_r\coloneqq\{\a\in\N^n\mid|\a|\le r\}$. We use $A\succeq0$ to indicate that the matrix $A$ is positive semidefinite (PSD).
Let $\i$ be the imaginary unit, satisfying $\i^2 = -1$.
Let $\z=(z_1,\ldots,z_n)$ be a tuple of complex variables and $\overline{\z}=(\overline{z}_1,\ldots,\overline{z}_n)$ its conjugate. Let $\overline{a}$ denote the conjugate of a complex number $a$ and $\bv^{*}$ (resp. $A^{*}$) denote the conjugate transpose of a complex vector $\bv$ (resp. a complex matrix $A$).
We denote by $\C[\z,\overline{\z}]\coloneqq\C[z_1,\ldots,z_n,\overline{z}_1,\ldots,\overline{z}_n]$, $\R[\z,\overline{\z}]\coloneqq\R[z_1,\ldots,z_n,\overline{z}_1,\ldots,\overline{z}_n]$, the complex and real polynomial rings in $\z,\overline{\z}$, respectively. A polynomial $f\in\C[\z,\overline{\z}]$ (resp. $\R[\z,\overline{\z}]$) can be written as $f=\sum_{(\b,\g)\in\N^n\times\N^n}f_{\b,\g}\z^{\b}\overline{\z}^{\g}$ with $f_{\b,\g}\in\C$ (resp. $f_{\b,\g}\in\R$), $\z^{\b}=z_1^{\beta_1}\cdots z_n^{\beta_n},\overline{\z}^{\g}=\overline{z}_1^{\gamma_1}\cdots \overline{z}_n^{\gamma_n}$. 
The \emph{conjugate} of $f$ is $\overline{f}=\sum_{(\b,\g)\in\N^n\times\N^n}\overline{f}_{\b,\g}\z^{\g}\overline{\z}^{\b}$. The polynomial $f$ is \emph{self-conjugate} if $\overline{f}=f$. It is clear that self-conjugate polynomials take only real values. 
The set of self-conjugate polynomials in $\C[\z,\overline{\z}]$ (resp. $\R[\z,\overline{\z}]$) is denoted by $\C[\z,\overline{\z}]^{\rm{c}}$ (resp. $\R[\z,\overline{\z}]^{\rm{c}}$). For example, we have $3-|z_1|^2+\frac{1+\i}{2}z_1\overline{z}_2^2+\frac{1-\i}{2}z_2^2\overline{z}_1\in\C[\z,\overline{\z}]^{\rm{c}}$ and $3-|z_1|^2+\frac{1}{2}z_1\overline{z}_2^2+\frac{1}{2}z_2^2\overline{z}_1\in\R[\z,\overline{\z}]^{\rm{c}}$. For $r\in\N$, let us denote
\begin{equation*}
    \C_r[\z,\overline{\z}]\coloneqq\left\{\sum_{|\b|,|\g|\le r}f_{\b,\g}\z^{\b}\overline{\z}^{\g}\middle| f_{\b,\g}\in\C\right\},
\end{equation*}
and
\begin{equation*}
    \R_r[\z,\overline{\z}]\coloneqq\left\{\sum_{|\b|,|\g|\le r}f_{\b,\g}\z^{\b}\overline{\z}^{\g}\middle| f_{\b,\g}\in\R\right\}. 
\end{equation*}
Furthermore, let $\C_r[\z,\overline{\z}]^{\rm{c}}\coloneqq\C[\z,\overline{\z}]^{\rm{c}}\cap\C_r[\z,\overline{\z}]$ and $\R_r[\z,\overline{\z}]^{\rm{c}}\coloneqq\R[\z,\overline{\z}]^{\rm{c}}\cap\R_r[\z,\overline{\z}]$.

A polynomial $\sigma=\sum_{(\b,\g)}\sigma_{\b,\g}\z^{\b}\overline{\z}^{\g}\in\C[\z,\overline{\z}]^{\rm{c}}$ is called an \emph{Hermitian sum of squares} (HSOS) if there exist polynomials $f_i\in\C[\z], i\in[t]$ such that $\sigma=|f_1|^2+\cdots+|f_t|^2$. The set of HSOS polynomials is denoted by $\Sigma^{\C}[\z,\overline{\z}]$. For $\sigma\in\R[\z,\overline{\z}]^{\rm{c}}$, one can further require that there exist polynomials $f_i\in\R[\z], i\in[t]$ such that $\sigma=|f_1|^2+\cdots+|f_t|^2$, and the set of such $\sigma$ is denoted by $\Sigma^{\R}[\z,\overline{\z}]$. We also denote $\Sigma_r^{\C}[\z,\overline{\z}]\coloneqq\Sigma^{\C}[\z,\overline{\z}]\cap\C_r[\z,\overline{\z}]$ and $\Sigma_r^{\R}[\z,\overline{\z}]\coloneqq\Sigma^{\R}[\z,\overline{\z}]\cap\R_r[\z,\overline{\z}]$.



\subsection{The complex moment-HSOS hierarchy}\label{complexsos}
By invoking Borel measures, \eqref{cpop} admits the following reformulation:
\begin{equation}\label{meas}
\begin{cases}
\inf\limits_{\mu\in\mathcal{M}_+(K)} &\int_{K} f\,\mathrm{d}\mu\\
\quad\,\,\rm{s.t.}&\int_{K}\,\mathrm{d}\mu=1,\\
\end{cases}
\end{equation}
where 
\begin{equation}
K\coloneqq\left\{\z\in\C^n\mid g_i(\z,\overline{\z})\ge0,i\in[m],\,h_i(\z,\overline{\z})=0,i\in[l]\right\},    
\end{equation}
and $\mathcal{M}_+(K)$ denotes the set of finite positive Borel measures on $K$. This reformulation leads to a complex moment hierarchy for \eqref{cpop}. To this end, let $\y=(y_{\b,\g})_{(\b,\g)\in\N^n\times\N^n}\subseteq\C$ be a pseudo-moment sequence indexed by $(\b,\g)\in\N^n\times\N^n$ satisfying $y_{\b,\g}=\overline{y}_{\g,\b}$. Let $L_{\y}:\C[\z,\overline{\z}]\rightarrow\R$ be the linear Riesz functional associated with $\y$ given by
\begin{equation*}
f=\sum_{(\b,\g)}f_{\b,\g}\z^{\b}\overline{\z}^{\g}\mapsto L_{\y}(f)=\sum_{(\b,\g)}f_{\b,\g}y_{\b,\g}.
\end{equation*}
For $r\in\N$, the \emph{(complex) moment} matrix $\M_{r}(\y)$ (of order $r$) associated with $\y$ is the matrix indexed by $\N^n_{r}$ such that
\begin{equation*}
[\M_r(\y)]_{\b\g}\coloneqq L_{\y}(\z^{\b}\overline{\z}^{\g})=y_{\b,\g}, \quad\forall\b,\g\in\N^n_{r}.
\end{equation*}
Let $g=\sum_{(\b',\g')}g_{\b',\g'}\z^{\b'}\overline{\z}^{\g'}\in\C[\z,\overline{\z}]^{\rm{c}}$. The \emph{(complex) localizing} matrix $\M_{r}(g\y)$ (of order $r$) associated with $g$ and $\y$ is the matrix indexed by $\N^n_{r}$ such that
\begin{equation*}
[\M_{r}(g\y)]_{\b\g}\coloneqq L_{\y}(g\z^{\b}\overline{\z}^{\g})=\sum_{(\b',\g')}g_{\b',\g'}y_{\b+\b',\g+\g'}, \quad\forall\b,\g\in\N^n_{r}.
\end{equation*}
Both the moment and localizing matrices are Hermitian matrices by definition.

Let $d^f\coloneqq\max\,\{|\b|,|\g|: f_{\b,\g}\ne0\}$, $d^g_i\coloneqq\max\,\{|\b|,|\g|: g^i_{\b,\g}\ne0\}$ for $i\in[m]$, $d^h_i\coloneqq\max\,\{|\b|,|\g|: h^i_{\b,\g}\ne0\}$ for $i\in[l]$, where $f=\sum_{(\b,\g)}f_{\b,\g}\z^{\b}\overline{\z}^{\g},g_i=\sum_{(\b,\g)}g^i_{\b,\g}\z^{\b}\overline{\z}^{\g},h_i=\sum_{(\b,\g)}h^i_{\b,\g}\z^{\b}\overline{\z}^{\g}$. Moreover, let
\begin{equation}
   d_{K}\coloneqq\max\left\{2,d^g_1,\ldots,d^g_m,d^h_1,\ldots,d^h_l\right\},
\end{equation}
and $d_{\min}\coloneqq\max\,\{d^f,d^g_1,\ldots,d^g_m,d^h_1,\ldots,d^h_l\}$.
The complex moment hierarchy indexed by $r\ge d_{\min}$ (called the relaxation order) for \eqref{cpop} is then given by
\begin{equation}\label{CMom}
\theta_r\coloneqq\begin{cases}
\inf\limits_{\y\subseteq\C}& L_{\y}(f)\\
\,\,\rm{s.t.}&\M_{r}(\y)\succeq0,\\
&\M_{r-d^g_i}(g_i\y)\succeq0,\quad i\in[m],\\
&\M_{r-d^h_i}(h_i\y)=0,\quad i\in[l],\\
&y_{\mathbf{0},\mathbf{0}}=1,\tag{CMom-$r$}
\end{cases}
\end{equation}
which is a complex semidefinite program (SDP). The dual of \eqref{CMom} can be formulized as follows:
\begin{equation}\label{CSOS}
\hat{\theta}_r\coloneqq\begin{cases}
\sup\limits_{\gamma,\sigma_i,\tau_i}&\gamma\\
\,\,\,\,\rm{s.t.}&f-\gamma=\sigma_0+\sigma_1g_1+\cdots+\sigma_mg_m+\tau_1h_1+\cdots+\tau_lh_l,\\
&\sigma_0\in\Sigma_{r}^{\C}[\z,\overline{\z}],\sigma_i\in\Sigma_{r-d^g_i}^{\C}[\z,\overline{\z}],i\in[m],\,\tau_i\in\C_{r-d^h_i}[\z,\overline{\z}]^{\rm{c}},i\in[l].\tag{CSOS-$r$}
\end{cases}
\end{equation}
To reformulate \eqref{CMom} or \eqref{CSOS} as an SDP over real numbers, we refer the reader to \cite{wang2023efficient}.

\begin{remark}
By the complex version of Putinar's Positivstellensatz theorem due to D’Angelo and Putinar \cite{d2009polynomial}, global asymptotic convergence of the complex moment-HSOS hierarchy is guaranteed when a sphere constraint is present.
\end{remark}

\section{The real moment-HSOS hierarchy}
Our starting point is the following theorem stating that a self-conjugate polynomial with real coefficients is an HSOS if and only if it is also an HSOS with real coefficients.
\begin{theorem}\label{thm1}
Let $f=\sum_{(\b,\g)}f_{\b,\g}\z^{\b}\overline{\z}^{\g}\in\R[\z,\overline{\z}]^{\rm{c}}$. Then it holds
\begin{equation}
    f\in\Sigma_d^{\C}[\z,\overline{\z}]\iff f\in\Sigma_d^{\R}[\z,\overline{\z}].
\end{equation}
\end{theorem}
\begin{proof}
We shall prove only the direction ``$\Rightarrow$'' as the opposite direction is trivial. Let $[\z]_d=(\z^{\a})_{|\a|\le d}$ be the standard complex monomial basis. Let $G$ be the symmetric real matrix indexed by $|\a|\le d$ satisfying $G_{\b,\g}=f_{\b,\g}$ such that $f=[\z]_d^*G[\z]_d$. Since $f\in\Sigma_d^{\C}[\z,\overline{\z}]$, there exists $F\in\C^{s\times t}$ ($t\coloneqq\binom{n+d}{d}$) such that $f=(F[\z]_d)^*(F[\z]_d)$. It follows $G=F^*F$. Let $\ba\in\R^{t}$ be arbitrary. We have $\ba^{\intercal}G\ba=\ba^{*}F^*F\ba=(F\ba)^*(F\ba)\ge0$. Hence $G$ is PSD and so $f\in\Sigma_d^{\R}[\z,\overline{\z}]$.
\end{proof}

Theorem \ref{thm1} can be further extended to Putinar type representations of positive polynomials on a semialgebraic set.
\begin{theorem}\label{thm2}
Let $f,g_1,\ldots,g_m,h_1,\ldots,h_l\in\R[\z,\overline{\z}]^{\rm{c}}$. Assume that $f$ admits a representation
\begin{equation}
    f=\sigma_0+\sum_{i=1}^m\sigma_i g_i+\sum_{i=1}^l\tau_i h_i,
\end{equation}
where $\sigma_i\in\Sigma_{s_i}^{\C}[\z,\overline{\z}],i\in\{0\}\cup[m]$ and $\tau_i\in\C_{t_i}[\z,\overline{\z}]^{\rm{c}},i\in[l]$. Then there exist $\sigma'_i\in\Sigma^{\R}_{s_i}[\z,\overline{\z}],i\in\{0\}\cup[m]$ and $\tau'_i\in\R_{t_i}[\z,\overline{\z}]^{\rm{c}},i\in[l]$ such that
\begin{equation}
    f=\sigma'_0+\sum_{i=1}^m\sigma'_i g_i+\sum_{i=1}^l\tau'_i h_i.
\end{equation}
\end{theorem}
\begin{proof}
Since $f,g_1,\ldots,g_m,h_1,\ldots,h_l\in\R[\z,\overline{\z}]^{\rm{c}}$, we have
\begin{align*}
f(\z,\overline{\z})&=\frac{1}{2}\left(f(\z,\overline{\z})+f(\overline{\z},\z)\right)&\\
&=\frac{1}{2}\left(\sigma_0(\z,\overline{\z})+\sigma_0(\overline{\z},\z)\right)+\sum_{i=1}^m\frac{1}{2}\left(\sigma_i(\z,\overline{\z})+\sigma_i(\overline{\z},\z)) g_i(\z,\overline{\z}\right)\\
&\quad\quad+\sum_{i=1}^l\frac{1}{2}\left(\tau_i(\z,\overline{\z})+\tau_i(\overline{\z},\z)) h_i(\z,\overline{\z}\right).
\end{align*}
Let $\sigma'_i=\frac{1}{2}(\sigma_i(\z,\overline{\z})+\sigma_i(\overline{\z},\z))\in\Sigma_{s_i}^{\C}[\z,\overline{\z}]$ for $i\in\{0\}\cup[m]$ and $\tau'_i=\frac{1}{2}(\tau_i(\z,\overline{\z})+\tau_i(\overline{\z},\z))\in\C_{t_i}[\z,\overline{\z}]^{\rm{c}}$ for $i\in[l]$. One can easily check that $\sigma'_i,\tau'_i\in\R[\z,\overline{\z}]^{\rm{c}}$ for all $i$. Thus by Theorem \ref{thm1}, we have $\sigma'_i\in\Sigma^{\R}_{s_i}[\z,\overline{\z}],i\in\{0\}\cup[m]$ and $\tau'_i\in\R_{t_i}[\z,\overline{\z}]^{\rm{c}},i\in[l]$, which completes the proof.
\end{proof}

Theorem \ref{thm2} allows us to consider the following alternative to the HSOS relaxation \eqref{CSOS} of \eqref{cpop}:
\begin{equation}\label{RSOS}
\hat{\eta}_r\coloneqq\begin{cases}
\sup\limits_{\gamma,\sigma_i,\tau_i}&\gamma\\
\,\,\,\,\rm{s.t.}&f-\gamma=\sigma_0+\sigma_1g_1+\cdots+\sigma_mg_m+\tau_1h_1+\cdots+\tau_lh_l,\\
&\sigma_0\in\Sigma_r^{\R}[\z,\overline{\z}],\sigma_i\in\Sigma_{r-d^g_i}^{\R}[\z,\overline{\z}],i\in[m],\,\tau_i\in\R_{r-d^h_i}[\z,\overline{\z}]^{\rm{c}},i\in[l].\tag{RSOS-$r$}
\end{cases}
\end{equation}

On the other hand, the assumption that the coefficients of $f,g_1,\ldots,g_m,h_1,\ldots,h_l$ are all real implies that \eqref{cpop} is invariant under the conjugate transformation: $\z\mapsto\overline{\z}$. Therefore, in \eqref{meas} it is more natural to seek a Borel measure that is invariant under the conjugate transformation. A conjugate-invariant Borel measure has real moments. Hence we can consider the following moment relaxation for \eqref{cpop} defined over real numbers:

\begin{equation}\label{RMom}
\eta_r\coloneqq\begin{cases}
\inf\limits_{\y\subseteq\R}& L_{\y}(f)\\
\,\,\rm{s.t.}&\M_{r}(\y)\succeq0,\\
&\M_{r-d^g_i}(g_i\y)\succeq0,\quad i\in[m],\\
&\M_{r-d^h_i}(h_i\y)=0,\quad i\in[l],\\
&y_{\mathbf{0},\mathbf{0}}=1.
\end{cases}\tag{RMom-$r$}
\end{equation}
\eqref{RMom} and \eqref{RSOS} form a pair of primal-dual SDPs.

\begin{theorem}\label{thm3}
For any $r\ge d_{\min}$, one has $\theta_r=\eta_r$ and $\hat{\theta}_r=\hat{\eta}_r$.
\end{theorem}
\begin{proof}
Suppose that $\y$ is a feasible solution of \eqref{RMom}. Clearly $\y$ is also a feasible solution of \eqref{CMom}. So $\theta_r\le\eta_r$. Conversely, suppose that $\y$ is a feasible solution of \eqref{CMom}. Then $\y'=\frac{\y+\overline{\y}}{2}$ is a feasible solution of \eqref{RMom} and $L_{\y'}(f)=\frac{1}{2}(L_{\y}(f)+L_{\overline{\y}}(f))=L_{\y}(f)$. It follows $\theta_r\ge\eta_r$ and we then conclude $\theta_r=\eta_r$.

The equality $\hat{\theta}_r=\hat{\eta}_r$ follows from Theorem \ref{thm2}.
\end{proof}

In Table \ref{tab:comp}, we list the sizes of SDPs involved in the real HSOS hierarchy (R-HSOS), the complex HSOS hierarchy (C-HSOS)\footnote{Here, the size of a complex SDP is measured after converting it to an equivalent real SDP; see \cite{wang2023efficient}.}, and the real SOS hierarchy (R-SOS) (i.e., the usual SOS hierarchy applied to the equivalent real POP) for \eqref{cpop}, respectively.

\begin{table}[htbp]
	\caption{Complexity comparison of different hierarchies for \eqref{cpop}. $n$: number of complex variables, $r$: relaxation order, $n_{\rm{sdp}}$: maximal size of PSD blocks, $m_{\rm{sdp}}$: number of affine constraints.}\label{tab:comp}
	\renewcommand\arraystretch{1.4}
	\centering
	\begin{tabular}{c|c|c|c}
	&R-HSOS&C-HSOS&R-SOS\\
	\hline
	$n_{\rm{sdp}}$&$\binom{n+r}{r}$&$2\binom{n+r}{r}$&$\binom{2n+r}{r}$\\
	\hline
        $m_{\rm{sdp}}$&$\frac{1}{2}\left(\binom{n+r}{r}+1\right)\binom{n+r}{r}$&$\binom{n+r}{r}^2$&$\binom{2n+2r}{2r}$
	\end{tabular}
\end{table}

\begin{remark}
If \eqref{cpop} additionally exhibits correlative and/or term sparsity, then one can construct a sparsity-adapted real moment-HSOS hierarchy in a similar spirit with \cite{wang2022exploiting}.
\end{remark}

\section{Detecting global optimality and extracting solutions}
In general, detecting global optimality and extracting solutions for \eqref{RMom} can be carried out in the same way as for \eqref{CMom}. For details, we refer the reader to Proposition 4.1 of \cite{josz2018lasserre} and the discussions therein. However, we observe that the ``joint hyponormality'' condition in \cite[Proposition 4.1]{josz2018lasserre} could be actually replaced with the following simpler ``hyponormality'' condition.

\begin{corollary}\label{cor:flat}
Consider \eqref{cpop} with complex (not necessarily real) coefficients.
Suppose that \eqref{cpop} involves a sphere constraint $|z_1|^2 + \dots + |z_n|^2 = R$ for some $R > 0$.
Let $\y$ be an optimal solution of \eqref{CMom}.
\begin{itemize}
\item If there is an integer $t$ such that $d_{\min} \leq t \leq r$ and $\rank\,\M_t(\y) = 1$, then  $\theta_r = f_{\min}$ and there is at least one global minimizer.
\item If there exists an integer $t$ such that $\max\,\{d_K, d_{\min}\} \leq t \leq r$ and the following two conditions hold
\begin{enumerate}
\item (Flatness) $\rank\,\M_t(\y) = \rank\,\M_{t-d_K}(\y)\,\,(\eqqcolon s)$;
\item (Hyponormality) For all $i\in [n]$,
$$\begin{bmatrix}
\M_{t-d_K}(\y) & \M_{t-d_K}(\overline{z_i}\y) \\
\M_{t-d_K}(z_i\y) & \M_{t-d_K}(|z_i|^2\y)
\end{bmatrix} \succeq0.$$
\end{enumerate}
Then $\theta_r = f_{\min}$ and there are at least $s$ global minimizers.
\end{itemize}
\end{corollary}
\begin{proof}
The proof is almost the same as the one of \cite[Proposition 4.1]{josz2018lasserre}, which relies on  \cite[Theorem 5.1]{josz2018lasserre}.
One just needs to note that the condition \emph{(Hyponormality)} implies that the so-called {\em shift} operators are all normal, which together with pairwise commutativity suffices to guarantee that these shift operators (and their adjoint operators) are simultaneously unitarily diagonalizable; see Theorem 2.5.5 of \cite{horn2012matrix} (and Theorem 1 of \cite{fuglede1950commutativity}).
\end{proof}

Next, we concentrate on a particular case when $\rank\,{\M_r(\y)}=2$, which frequently occurs in practice. We will see that the rank-two condition allows us to remove the hyponormality assumption in Corollary \ref{cor:flat}. To begin with, let us prove the following proposition which is crucial for the proof of our main theorem and may also have independent interests \cite{piekarz2023neumann}.
\begin{proposition}\label{sec4:thm1}
Suppose that the matrices $A_1,\ldots,A_n\in\R^{2\times2}$ commute pairwise and the equality
\begin{equation}\label{sec4:eq1}
    A_1^{\intercal}A_1+\cdots+A_n^{\intercal}A_n=RI_2
\end{equation}
holds for some $R>0$. Here, $I_2$ denotes the $2\times2$ identity matrix. Then one of the following two statements holds true:
\begin{enumerate}[(1)]
    \item 
    $A_i=\left[\begin{smallmatrix}a_i&b_i\\b_i&d_i\end{smallmatrix}\right]$ for $i=1,\ldots,n$;
    \item
    $A_i=\left[\begin{smallmatrix}a_i&-b_i\\b_i&a_i\end{smallmatrix}\right]$ for $i=1,\ldots,n$.
\end{enumerate}
As a result, $A_1,\ldots,A_n$ are simultaneously unitarily diagonalizable.
\end{proposition}
\begin{proof}
Let $A_i=\left[\begin{smallmatrix}a_i&c_i\\b_i&d_i\end{smallmatrix}\right]$ for $i=1,\ldots,n$. Since $A_1,\ldots,A_n$ commute pairwise, we have
\begin{subequations}
\begin{align}
   a_ib_j+b_id_j&=a_jb_i+b_jd_i,\label{sec4:eq2a}\\
   a_ic_j+c_id_j&=a_jc_i+c_jd_i,\label{sec4:eq2b}\\
   b_ic_j&=b_jc_i,\label{sec4:eq2c}
\end{align}
\end{subequations}
for all $1\le i<j\le n$. Moreover, the equality \eqref{sec4:eq1} yields
\begin{subequations}
\begin{align}
   \sum_{i=1}^n\left(a_i^2+b_i^2\right)&=R,\label{sec4:eq3a}\\
   \sum_{i=1}^n\left(c_i^2+d_i^2\right)&=R,\label{sec4:eq3b}\\
   \sum_{i=1}^n\left(a_ic_i+b_id_i\right)&=0.\label{sec4:eq3c}
\end{align}
\end{subequations}
First we consider the case that not all $c_i$ are zeros (i.e., $\sum_{i=1}^nc_i^2\ne0$).
By \eqref{sec4:eq2c}, we may assume $b_i=tc_i$ for some $t$ and for $i=1,\ldots,n$. By \eqref{sec4:eq2b}, we may assume $a_i=c_is+d_i$ for some $s$ and for $i=1,\ldots,n$. Then \eqref{sec4:eq3a} and \eqref{sec4:eq3c} become
\begin{subequations}
\begin{align}
    \sum_{i=1}^n\left(c_i^2s^2+2c_id_is+d_i^2+t^2c_i^2\right)&=R,\label{sec4:eq4a}\\
   \sum_{i=1}^n\left(c_i^2s+c_id_i+tc_id_i\right)&=0.\label{sec4:eq4b}
\end{align}
\end{subequations}
\eqref{sec4:eq4a}$-$\eqref{sec4:eq4b}$\times s-$\eqref{sec4:eq3b} yields
\begin{equation}\label{sec4:eq5}
    (1-t)\left(s\sum_{i=1}^nc_id_i-(1+t)\sum_{i=1}^nc_i^2\right)=0.
\end{equation}
It follows from \eqref{sec4:eq5} that one of the following cases holds true: (i) $t=1$, which implies that $A_i$ are all of form in (1); (ii) $t=-1$ and then by \eqref{sec4:eq4b} $s=0$, which implies that $A_i$ are all of form in (2); (iii) $t\ne\pm1$ and
\begin{equation}\label{sec4:eq6}
    s=\frac{(1+t)\sum_{i=1}^nc_i^2}{\sum_{i=1}^nc_id_i}.
\end{equation}
For case (iii), substituting into \eqref{sec4:eq4a} with \eqref{sec4:eq6}, we obtain
\begin{align*}
&\,\sum_{i=1}^n\left(c_i^2s^2+2c_id_is+d_i^2+t^2c_i^2\right)-R\\
=&\,\frac{(1+t)^2\left(\sum_{i=1}^nc_i^2\right)^3}{\left(\sum_{i=1}^nc_id_i\right)^2}+2(1+t)\sum_{i=1}^nc_i^2+\sum_{i=1}^nd_i^2+t^2\sum_{i=1}^nc_i^2-R\\
=&\,\frac{(1+t)^2\left(\sum_{i=1}^nc_i^2\right)^3}{\left(\sum_{i=1}^nc_id_i\right)^2}+(1+t)^2\left(\sum_{i=1}^nc_i^2\right)\\
=&\,(1+t)^2\left(\sum_{i=1}^nc_i^2\right)\left(\frac{\left(\sum_{i=1}^nc_i^2\right)^2}{\left(\sum_{i=1}^nc_id_i\right)^2}+1\right)=0,
\end{align*}
a contradictory.

Now suppose $c_i=0$ for all $i$. Then we must have $b_i=0$ for all $i$ and so $A_i$ fits into (1). Indeed, if not all $b_i$ are zeros, then by \eqref{sec4:eq2a}, we may assume $a_i=b_is+d_i$ for some $s$ and for $i=1,\ldots,n$. \eqref{sec4:eq3b} and \eqref{sec4:eq3c} become
$\sum_{i=1}^nd_i^2=R$ and $\sum_{i=1}^nb_id_i=0$. Thus \eqref{sec4:eq3a} yields
\begin{equation*}
\sum_{i=1}^n\left(b_i^2s^2+2b_id_is+d_i^2+b_i^2\right)-R=(1+s^2)\sum_{i=1}^nb_i^2=0,
\end{equation*}
which implies $\sum_{i=1}^nb_i^2=0$, a contradictory. 
\end{proof}

Now we are ready to prove the promised theorem.
\begin{theorem}\label{sec4:thm2}
Suppose that \eqref{cpop} involves a sphere constraint $|z_1|^2 + \cdots + |z_n|^2=R$ for some $R > 0$, and let $\y$ be an optimal solution of \eqref{RMom}. 
If $\rank\,{\M_r(\y)}=\rank\,{\M_{r-d_K}(\y)}=2$, then global optimality is attained, i.e., $\theta_r=f_{\min}$. Moreover, one can extract either two real solutions or a pair of complex conjugate solutions for \eqref{cpop}.
\end{theorem}
\begin{proof}
The proof is an adaptation of the one of \cite[Theorem 5.1]{josz2018lasserre}. 

The real PSD matrix $\M_r(\y)$ of rank $2$ can be factorized in the Grammian form such that $y_{\b, \g} = a_{\b}^{\intercal}a_{\g}$ for all $|\b|,|\g| \leq r$, where $a_{\b}\in\R^2$. Since $\rank\,{\M_{r-1}(\y)}=2$, we have $\R^2=\Span(a_{\a})_{|\a|\le r-1}$. Let us consider the shift operators\footnote{For convenience, we identify the operators with their matrices in the standard basis of $\R^2$.} $T_1 ,\ldots,T_n : \R^2 \to \R^2$, given by 
\begin{equation}
    T_i:\sum_{|\a| \leq r-1} w_{\a} a_{\a}\mapsto \sum_{|\a| \leq r-1} w_{\a} a_{\a+\e_i},\quad i\in[n],
\end{equation}
where $\{\e_1,\ldots,\e_n\}$ is the standard vector basis of $\R^n$.
To ensure that the shifts are well-defined, we have to check that each element of $\R^2$ has a unique image under $T_i$, that is, if $\sum_{|\a| \leq r-1}u_{\a}a_{\a}=\sum_{|\a| \leq r-1}v_{\a}a_{\a}$, then $\sum_{|\a| \leq r-1}u_{\a}a_{\a+\e_i}=\sum_{|\a| \leq r-1}v_{\a}a_{\a+\e_i}$ for all $i\in[n]$. Actually, the sphere constraint yields
\begin{equation*}
    \sum_{i=1}^ny_{\b+\e_i,\g+\e_i}=Ry_{\b,\g}
\end{equation*}
for all $|\b|,|\g|\le r-1$. Thus,
\begin{align*}
\sum_{i=1}^n\left\|\sum_{|\a| \leq r-1}w_{\a}a_{\a+\e_i}\right\|^2&=\sum_{i=1}^n\sum_{|\b|,|\g| \leq r-1}w_{\b}w_{\g}a_{\b+\e_i}a^{\intercal}_{\g+\e_i}\\
&=\sum_{|\b|,|\g| \leq r-1}w_{\b}w_{\g}\sum_{i=1}^ny_{\b+\e_i,\g+\e_i}\\
&=R\sum_{|\b|,|\g| \leq r-1}w_{\b}w_{\g}y_{\b,\g}\\
&=R\left\|\sum_{|\a| \leq r-1}w_{\a}a_{\a}\right\|^2.
\end{align*}
It follows that if $\sum_{|\a| \leq r-1}w_{\a}a_{\a}=0$, then $\sum_{|\a| \leq r-1}w_{\a}a_{\a+\e_i}=0$ for all $i\in[n]$, from which we immediately obtain the desired result.

We now prove that $T_1,\ldots,T_n$ commute pairwise. First note that for any $\a\in\N^n_{r-2}$, we have $T_iT_j(a_{\a})=T_i(a_{\a+\e_j})=a_{\a+\e_i+\e_j}=a_{\a+\e_j+\e_i}=T_jT_i(a_{\a})$. Therefore, for any $u\in\R^2$ with $u=\sum_{|\a| \leq r-2}u_{\a}a_{\a}$ (this is always possible as $\rank\,{\M_{r-2}(\y)}=2$), we have
\begin{align*}
    T_iT_j(u)&=T_iT_j\left(\sum_{|\a| \leq r-1}u_{\a}a_{\a}\right)=\sum_{|\a| \leq r-1}u_{\a}T_iT_j(a_{\a})\\
    &=\sum_{|\a| \leq r-1}u_{\a}T_jT_i(a_{\a})=T_jT_i(u),
\end{align*}
which proves the pairwise commutativity of $T_1,\ldots,T_n$.
Moreover, we have
\begin{align*}
\sum_{i=1}^nu^{\intercal}T_i^{\intercal}T_iu&=\sum_{i=1}^n\sum_{|\b|,|\g|\le r-2}u_{\b}u_{\g}a_{\b+\e_i}^{\intercal}a_{\g+\e_i}\\
&=\sum_{|\b|,|\g|\le r-2}u_{\b}u_{\g}\sum_{i=1}^ny_{\b+\e_i,\g+\e_i}\\
&=R\sum_{|\b|,|\g|\le r-2}u_{\b}u_{\g}y_{\b,\g}\\
&=R\sum_{|\b|,|\g|\le r-2}u_{\b}u_{\g}a_{\b}^{\intercal}a_{\g}=Ru^{\intercal}u,
\end{align*}
which implies
\begin{equation*}
    T_1^{\intercal}T_1+\cdots+T_n^{\intercal}T_n=RI_2.
\end{equation*}
Then by Proposition \ref{sec4:thm1}, $T_1 ,\ldots,T_n$ are simultaneously unitarily diagonalizable. That is, there exists a unitary matrix $P$ such that $T_i = PD_iP^*$, $i=1,\ldots,n$, where $D_i=\diag\,(d_{i1},d_{i2})$ is a diagonal matrix. For convenience, we let $T^{\a}\coloneqq T_1^{\alpha_1}\cdots T_n^{\alpha_n}$. Then for $|\b|,|\g|\le r$, we have
\begin{align*}
y_{\b,\g}&=a_{\b}^{\intercal}a_{\g}=\left(T^{\b}a_{\mathbf{0}}\right)^{\intercal}\left(T^{\g}a_{\mathbf{0}}\right)=a_{\mathbf{0}}^{\intercal}\left(T^{\b}\right)^{\intercal}T^{\g}a_{\mathbf{0}}\\
&=a_{\mathbf{0}}^{\intercal}\left(PD^{\b}P^*\right)^{*}\left(PD^{\g}P^*\right)a_{\mathbf{0}}=a_{\mathbf{0}}^{\intercal}\left(\sum_{j=1}^2p_j\overline{d}_j^{\b}d_j^{\g}p_j^*\right)a_{\mathbf{0}}\\
&=\sum_{j=1}^2a_{\mathbf{0}}^{\intercal}p_jp_j^*a_{\mathbf{0}}\overline{d}_j^{\b}d_j^{\g}=\sum_{j=1}^2|a_{\mathbf{0}}^{\intercal}p_j|^2\overline{d}_j^{\b}d_j^{\g},
\end{align*}
where $p_1,p_2$ denote the columns of $P$ and $d_j\coloneqq(d_{1j},\ldots, d_{nj})$. As a result,
we obtain that the measure $\mu=|a_{\mathbf{0}}^{\intercal}p_1|^2\delta_{d_1}+|a_{\mathbf{0}}^{\intercal}p_2|^2\delta_{d_2}$ satisfies
$y_{\b,\g}=\int_{\C^n}\z^{\b}\overline{\z}^{\g}\,\mathrm{d}\mu$ for all $|\b|,|\g|\le r$.

Finally, we check whether the measure $\mu$ is supported on $K$. Indeed, we have
\begin{align*}
g_i(d_j,\overline{d}_j)&=p_j^*p_j\sum_{\b,\g}g^i_{\b,\g}d_j^{\b}\overline{d}_j^{\g} \quad\quad(\text{since } p_j^*p_j=1)\\
&=\sum_{\b,\g}g^i_{\b,\g}\left(d_j^{\g}p_j\right)^*\left(d_j^{\b}p_j\right)\\
&=\sum_{\b,\g}g^i_{\b,\g}\left(T^{\g}p_j\right)^*\left(T^{\b}p_j\right)\quad\quad(\text{let } p_j=\sum_{|\a|\le r-d_K}p^j_{\a}a_{\a})\\
&=\sum_{|\b'|,|\g'|\le r-d_K}p^j_{\b'}\overline{p}^j_{\g'}\left(\sum_{\b,\g}g^i_{\b,\g}\left(T^{\g}a_{\g'}\right)^{\intercal}\left(T^{\b}a_{\b'}\right)\right)\\
&=\sum_{|\b'|,|\g'|\le r-d_K}p^j_{\b'}\overline{p}^j_{\g'}\left(\sum_{\b,\g}g^i_{\b,\g}a_{\g'+\g}^{\intercal}a_{\b'+\b}\right)\\
&=\sum_{|\b'|,|\g'|\le r-d_K}p^j_{\b'}\overline{p}^j_{\g'}\left(\sum_{\b,\g}g^i_{\b,\g}y_{\b'+\b,\g'+\g}\right)\ge0.
\end{align*}
The last inequality is because $\M_{r-d^g_i}(g_i\y)\succeq0$ and $r-d_K\le r-d^g_i$. Similarly, we can show $h_i(d_j,\overline{d}_j)=0$ by using $\M_{r-d^h_i}(h_i\y)=0$.

By Proposition \ref{sec4:thm1}, either all $T_i$ are symmetric matrices or all $T_i$ are matrices of similarity transformations. In the former case, two real solutions for \eqref{cpop} can be extracted; in the latter case, a pair of complex conjugate solutions for \eqref{cpop} can be extracted.
\end{proof}


The sphere constraint in Theorem \ref{sec4:thm2} can be replaced by the unit-norm constraint on each variable as the latter together with $\rank\,{\M_r(\y)}=\rank\,{\M_{r-2}(\y)}$ can also ensure that $T_1,\ldots,T_n$ are simultaneously unitarily diagonalizable; see \cite[Appendix B]{josz2018lasserre}.
\begin{theorem}\label{sec4:thm3}
Suppose that \eqref{cpop} involves the unit-norm constraint $|z_i|^2=1$ for all $i\in[n]$, and let $\y$ be an optimal solution of \eqref{RMom}. 
If $\rank\,{\M_r(\y)}=\rank\,{\M_{r-d_K}(\y)}=2$, then global optimality is attained, i.e., $\theta_r=f_{\min}$. Moreover, one can extract either two real solutions or a pair of complex conjugate solutions for \eqref{cpop}.
\end{theorem}

\begin{remark}
The condition $\rank\,{\M_r(\y)}=\rank\,{\M_{r-d_K}(\y)}=2$ in Theorems \ref{sec4:thm2} and \ref{sec4:thm3} can be weakened as: there exists an integer $t$ such that $\max\,\{d_K, d_{\min}\} \leq t \leq r$ such that $\rank\,\M_t(\y) = \rank\,\M_{t-d_K}(\y)=2$.
\end{remark}

Theorems \ref{sec4:thm2} and \ref{sec4:thm3} guarantee that \eqref{cpop} has either two real solutions or a pair of complex conjugate solutions. For the latter case, we provide a simple procedure for extracting a pair of conjugate optimal solutions (denoted by $\z,\overline{\z}$) from the first-order moment matrix $\M_{1}(\y)$. First of all, note that as all $T_i$ are matrices of similarity transformations, we can take the unitary matrix $P=\left[\begin{smallmatrix}\frac{1}{\sqrt{2}}&\frac{1}{\sqrt{2}}\\\frac{\i}{\sqrt{2}}&-\frac{\i}{\sqrt{2}}\end{smallmatrix}\right]$ such that $T_i = PD_iP^*$, $i=1,\ldots,n$. Let $a_{\mathbf{0}}=[a^1_{\mathbf{0}},a^2_{\mathbf{0}}]^{\intercal}$ and we have $a_{\mathbf{0}}^{\intercal}a_{\mathbf{0}}=y_{\mathbf{0},\mathbf{0}}=1$. 
Then $\left|a_{\mathbf{0}}^{\intercal}p_1\right|^2=\left|\frac{1}{\sqrt{2}}a^1_{\mathbf{0}}+\frac{\i}{\sqrt{2}}a^2_{\mathbf{0}}\right|^2=\frac{1}{2}a_{\mathbf{0}}^{\intercal}a_{\mathbf{0}}=\frac{1}{2}$. Similarly, $\left|a_{\mathbf{0}}^{\intercal}p_2\right|^2=\frac{1}{2}$. Therefore, $\mu=\frac{1}{2}\delta_{\z}+\frac{1}{2}\delta_{\overline{\z}}$.
Let $[\z]_1=[1,z_1,\ldots,z_n]^{\intercal}$. 
We then have $\M_{1}(\y)=\frac{1}{2}[\z]_1[\z]_1^{*}+\frac{1}{2}[\overline{\z}]_1[\overline{\z}]_1^{*}=\RR([\z]_1[\z]_1^{*})$ ($\RR$ means taking the real part). Let $\z=(\rho_i(\cos(\phi_i)+\sin(\phi_i)))_{i\in[n]}$. Because
\begin{align*}
    [\z]_1[\z]_1^{*}=\begin{bmatrix}
    1&\overline{z}_1&\cdots&\overline{z}_n\\
    z_1&|z_1|^2&\cdots&z_1\overline{z}_n\\
    \vdots&\vdots&\ddots&\vdots\\
    z_n&z_n\overline{z}_1&\cdots&|z_n|^2\\
    \end{bmatrix},
\end{align*}
we obtain
\begin{align*}
    \M_{1}(\y)&=\begin{bmatrix}
    1&\rho_1\cos(\phi_1)&\cdots&\rho_n\cos(\phi_n)\\
    \rho_1\cos(\phi_1)&\rho_1^2&\cdots&\rho_1 \rho_n\cos(\phi_1-\phi_n)\\
    \vdots&\vdots&\ddots&\vdots\\
    \rho_n\cos(\phi_n)&\rho_1 \rho_n\cos(\phi_1-\phi_n)&\cdots&\rho_n^2\\
    \end{bmatrix}\\
    &=\begin{bmatrix}
    1&0\\
    \rho_1\cos(\phi_1)& \rho_1\sin(\phi_1)\\
    \vdots&\vdots\\
    \rho_n\cos(\phi_n)&\rho_n\sin(\phi_n)\\
    \end{bmatrix}\cdot\begin{bmatrix}
    1& \rho_1\cos(\phi_1)&\cdots&\rho_n\cos(\phi_n)\\
    0&\rho_1\sin(\phi_1)&\cdots&\rho_n\sin(\phi_n)\\
    \end{bmatrix}.\\
\end{align*}
Thus, we can retrieve the pair of conjugate optimal solutions from the above ``Cholesky decomposition'' of $\M_{1}(\y)$.

\begin{example}
Let 
\begin{align*}
f(\z,\overline{\z})=&\,\,0.5z_1\overline{z}_2 + 0.5z_1\overline{z}_3 + 0.5z_2\overline{z}_1 + 0.25|z_2|^2 + 0.25z_2\overline{z}_3 + 0.5z_3\overline{z}_1 \\
&\,+ 0.25z_3\overline{z}_2+ z_1 + z_2 + z_3 + \overline{z}_1 + \overline{z}_2 + \overline{z}_3,
\end{align*}
and consider the CPOP: 
\begin{equation}\label{ex}
\inf_{\z\in\C^3} \left\{f(\z,\overline{\z})\,\,\,\hbox{\rm{s.t.}}\,\,\,|z_i|^2=1,\,i=1,2,3\right\}.
\end{equation}
The first-order moment-HSOS relaxation yields the value $-3.75$ and
\begin{equation*}
\M_1(\y)=\begin{bmatrix}
1     &  -0.250013 & -0.875003&  -0.875003\\
 -0.250013 &  1     &  -0.249981 & -0.249981\\
 -0.875003  &-0.249981 &  1  &      1\\
 -0.875003  &-0.249981  & 1 &       1
\end{bmatrix}.
\end{equation*}
We have
\begin{equation*}
\M_1(\y)=\begin{bmatrix}
1     &  0\\
 -0.250013 & 0.968242\\
 -0.875003  &-0.484117\\
 -0.875003  &-0.484117
\end{bmatrix}\cdot
\begin{bmatrix}
1&-0.250013&-0.875003&-0.875003\\
0& 0.968242&-0.484117&-0.484117
\end{bmatrix}.
\end{equation*}
Let
\begin{equation*}
  \z=(-0.250013+0.968242\i,-0.875003-0.484117\i,-0.875003-0.484117\i).
\end{equation*}
One can check that $\z$ is feasible to \eqref{ex} and $f(\z,\overline{\z})=-3.75$. Therefore, $\z$ and $\overline{\z}$ are a pair of conjugate optimal solutions to \eqref{ex}. Note that here we successfully extract two solutions even without the validity of $\rank\,{\M_r(\y)}=\rank\,{\M_{r-d_K}(\y)}$.
\end{example}

\section{Numerical experiments}\label{expe}
The real moment-HSOS hierarchy has been implemented in the Julia package {\tt TSSOS}\footnote{{\tt TSSOS} is freely available at \href{https://github.com/wangjie212/TSSOS}{https://github.com/wangjie212/TSSOS}.}. In this section, we compare the performance of the real moment-HSOS hierarchy (R-HSOS), the complex moment-HSOS hierarchy (C-HSOS), and the real moment-SOS hierarchy (R-SOS) on CPOPs with real coefficients using {\tt TSSOS} where {\tt Mosek} 10.0 \cite{mosek} is employed as an SDP solver with default settings. For R-SOS, we have exploited sign symmetry to reduce the SDP size \cite{tssos1}. In presenting the results, the column labelled by `opt' records optima of SDPs and the column labelled by `time' records running time in seconds. Moreover, the symbol `-' means that {\tt Mosek} runs out of memory.

\hspace{0.5em}

$\bullet$ {\bf Minimizing a random quadratic polynomial over the unit sphere.}
Our first example is to minimize a random quadratic polynomial over the unit sphere:
\begin{equation}\label{random:eq1}
\begin{cases}
\inf\limits_{\z\in\C^{n}} &[\z]_1^*Q[\z]_1\\
\,\,\,\rm{s.t.}&|z_1|^2+\cdots+|z_n|^2=1,
\end{cases}
\end{equation}
where $[\z]_1=(\z^{\a})_{|\a|\le 1}$ and $Q$ is a random real symmetric matrix whose entries are selected with respect to the uniform probability distribution on $[-1,1]$.

We approach \eqref{random:eq1} for $n=50,100,\ldots,300$ with the first-order relaxation of the three hierarchies. To measure the quality of obtained lower bounds, we also compute an upper bound (indicated by `ub' in the table) for \eqref{random:eq1} using the nonlinear programming solver {\tt Ipopt}.
The related results are shown in Table \ref{random:tab1}, from which we see that R-HSOS is the fastest and is more scalable with the size of the problem. 
Interestingly, we can also see that the lower bounds produced by the SDP relaxations coincide with the corresponding upper bounds, and thus are actually globally optimal.

\begin{table}[htbp]
	\caption{Minimizing a random quadratic polynomial over the unit sphere ($r=1$).}\label{random:tab1}
	\renewcommand\arraystretch{1.2}
	\centering
	\begin{tabular}{c|c|c|c|c|c|c|c}
		\multirow{2}{*}{$n$}&\multirow{2}{*}{ub}&\multicolumn{2}{c|}{R-HSOS}&\multicolumn{2}{c|}{C-HSOS}&\multicolumn{2}{c}{R-SOS}\\
		\cline{3-8}
		&&opt&time&opt&time&opt&time\\
		\hline
		50&-4.2477&-4.2477&0.12&-4.2477&0.79&-4.2477&0.36\\
		\hline
		100&-6.3120&-6.3120&2.90&-6.3120&17.3&-6.3120&8.52\\
		\hline
		150&-7.0808&-7.0808&21.2&-7.0808&134&-7.0808&66.1\\
		\hline
		200&-8.5434&-8.5434&104&-8.5434&731&-8.5434&324\\
		\hline
		250&-8.7356&-8.7356&421&-&-&-8.7356&1082\\
		\hline
		300&-9.7941&-9.7941&1163&-&-&-&-\\
	\end{tabular}
\end{table}

\hspace{0.5em}
$\bullet$ {\bf Minimizing a random quartic polynomial over the unit sphere.} Our second example is to minimize a quartic polynomial over the unit sphere:
\begin{equation}\label{random:eq2}
\begin{cases}
\inf\limits_{\z\in\C^{n}} &[\z]_2^*Q[\z]_2\\
\,\,\,\rm{s.t.}&|z_1|^2+\cdots+|z_n|^2=1,
\end{cases}
\end{equation}
where $[\z]_2=(\z^{\a})_{|\a|\le 2}$ and $Q$ is a random real symmetric matrix whose entries are selected with respect to the uniform probability distribution on $[-1,1]$.

We approach \eqref{random:eq2} for $n=8,10,\ldots,22$ with the second-order relaxations of the three hierarchies. To measure the quality of obtained lower bounds, we also compute an upper bound (indicated by `ub' in the table) for \eqref{random:eq2} using {\tt Ipopt}.
The related results are shown in Table \ref{random:tab2}, from which we see that R-HSOS is the fastest and is more scalable with the size of the problem.
Besides, we can see that the lower bounds produced by R-SOS coincide with the corresponding upper bounds, and thus are globally optimal, while the lower bounds produced by R-HSOS and C-HSOS are not globally optimal.

\begin{table}[htbp]
	\caption{Minimizing a random quartic polynomial over the unit sphere ($r=2$).}\label{random:tab2}
	\renewcommand\arraystretch{1.2}
	\centering
	\begin{tabular}{c|c|c|c|c|c|c|c}
		\multirow{2}{*}{$n$}&\multirow{2}{*}{ub}&\multicolumn{2}{c|}{R-HSOS}&\multicolumn{2}{c|}{C-HSOS}&\multicolumn{2}{c}{R-SOS}\\
		\cline{3-8}
		&&opt&time&opt&time&opt&time\\
		\hline
		8&-2.4039&-3.5594&0.10&-3.5594&0.51&-2.4039&0.78\\
		\hline
		10&-2.8058&-4.3101&0.43&-4.3101&3.10&-2.8058&4.21\\
		\hline
		12&-2.8804&-5.0127&2.00&-5.0127&11.1&-2.8804&19.7\\
		\hline
		14&-3.2747&-5.7066&6.59&-5.7066&50.9&-3.2747&92.4\\
		\hline
		16&-3.6565&-6.3434&27.0&-6.3434&180&-3.6565&342\\
		\hline
		18&-4.2988&-7.3378&91.3&-7.3378&596&-4.2988&1216\\
		\hline
		20&-4.1186&-7.4586&269&-7.4586&1844&-&-\\
            \hline
		22&-4.5440&-8.3169&717&-&-&-&-\\ 
	\end{tabular}
\end{table}

\hspace{0.5em}

$\bullet$ {\bf Smale’s Mean Value conjecture.} The next example is adapted from Smale’s Mean Value conjecture. 
Let us begin with a result of Smale \cite{smale1981fundamental}.
\begin{theorem}[\cite{smale1981fundamental}]\label{thm:smale}
Let $P\in\C[\z]$ be a polynomial of degree $n+1\ge2$ and $a\in\C$ such that $P'(a)\ne0$. Then there exists a critical point $b$ of $P$ such
that
\begin{equation}\label{eq:smale}
    \left|\frac{P(a)-P(b)}{a-b}\right|\le 4|P'(a)|.
\end{equation}
\end{theorem}
Theorem \ref{thm:smale} was proved by Smale in 1981, and he further asked whether one can replace the factor $4$ in \eqref{eq:smale} by 1, or even possibly by $\frac{n}{n+1}$, which is known as Smale’s Mean Value conjecture and also appears in \cite{smale1998mathematical}. This conjecture is verified for $n=1,2,3$ \cite[Chapter 7.2]{rahman2002analytic} and has the following equivalent (normalized) form.

\begin{conjecture}\label{conj:smale}
Let $P\in\C[\z]$ be a polynomial of degree $n+1\ge2$ such that $P(0)=0$ and $P'(0)=1$, and let $b_1, \ldots, b_{n}$ be its critical points. Then,
\begin{equation}\label{eq2:smale}
    \min\,\left\{\left|\frac{P(b_i)}{b_i}\right|:i=1,\ldots,n\right\}\le \frac{n}{n+1}.
\end{equation}
\end{conjecture}
The upper bound $\frac{n}{n+1}$ in Conjecture \ref{conj:smale} is attained by the polynomial $P(z)=z^{n+1}+z$.

We can reformulate Conjecture \ref{conj:smale} as the following CPOP by introducing a new variable $u$:
\begin{equation}\label{smale:eq1}
\begin{cases}
\sup\limits_{(\z,u)\in\C^{n+1}} &|u|\\
\,\,\quad\rm{s.t.}&|H(z_i)|\ge|u|,\quad i=1,\ldots,n,\\
&z_1\cdots z_n=\frac{(-1)^n}{n+1},\\
\end{cases}
\end{equation}
where $H(y)\coloneqq\frac{1}{y}\int_{0}^{y}p(z)\,\mathrm{d}z$ and $p(z)\coloneqq(n+1)(z-z_1)\cdots(z-z_n)$ with $p(0)=1$.
The optimum of \eqref{smale:eq1} is conjectured to be $\frac{n}{n+1}$. To ensure global convergence of the moment-HSOS hierarchies, we add a sphere constraint to \eqref{smale:eq1}:

\begin{equation}\label{smale:eq2}
\begin{cases}
\sup\limits_{(\z,u)\in\C^{n+1}} &|u|\\
\,\,\quad\rm{s.t.}&|H(z_i)|\ge|u|,\quad i=1,\ldots,n,\\
&z_1\cdots z_n=\frac{(-1)^n}{n+1},\\
&|z_1|^2+|z_2|^2+\cdots+|z_n|^2=n\left(\frac{1}{n+1}\right)^{\frac{2}{n}}.
\end{cases}
\end{equation}

We approach \eqref{smale:eq2} for $n=2,3,4,5$ with R-HSOS, C-HSOS, and R-SOS. To measure the quality of obtained upper bounds, we also compute an lower bound (indicated by `lb' in the table) for \eqref{smale:eq2} using {\tt Ipopt}. The related results are shown in Table \ref{smale:tab}. For $n=2$, the $2$nd order relaxation (numerically) retrieves the desired value $\frac{2}{3}$; for $n=3$, the $5$th order relaxation (numerically) retrieves the desired value $\frac{3}{4}$; for $n=4$, the $9$th order relaxation (numerically) retrieves the desired value $\frac{4}{5}$; for $n=5$, the hierarchies are not able to retrieve the desired value $\frac{5}{6}$ due to memory limitation. It can be seen from the table that R-HSOS is more scalable with the size of the problem. Note that for $n\ge4$, the local solver fails in returning lower bounds.

\begin{table}[htbp]
	\caption{Results for Smale’s Mean Value conjecture.}\label{smale:tab}
	\renewcommand\arraystretch{1.2}
	\centering
     \begin{tabular}{c|c|c|c|c|c|c|c|c}
		\multirow{2}{*}{$n$}&\multirow{2}{*}{$r$}&\multirow{2}{*}{lb}&\multicolumn{2}{c|}{R-HSOS}&\multicolumn{2}{c|}{C-HSOS}&\multicolumn{2}{c}{R-SOS}\\
		\cline{4-9}
		&&&opt&time&opt&time&opt&time\\
		\hline
		2&2&0.66666&0.66666&0.005&0.66666&0.006&0.66666&0.009\\
		\hline
		3&5&0.75000&0.75000&0.15&0.75000&0.45&0.75000&36.6\\
		\hline
		4&9&$*$&0.80000&6257&-&-&-&-\\
		\hline
		5&7&$*$&1.20615&16497&-&-&-&-\\
	\end{tabular}
\end{table}

\hspace{0.5em}

$\bullet$ {\bf The Mordell inequality conjecture.} Our final example is the Mordell inequality conjecture. In 1958, Birch proposed the following conjecture: Suppose that the complex numbers $z_1,\ldots,z_n\in\C$ satisfies $|z_1|^2+\cdots+|z_n|^2=n$. Then the maximum of $\Delta\coloneqq\prod_{1\le i<j\le n}|z_i-z_j|^2$ is $n^n$. In 1960, Mordell proved affirmatively the conjecture for $n=3$ \cite{mordell1960discriminant}. In the same year, J. H. H. Chalk gave a counterexample in \cite{chalk1960note} which disproves the conjecture for $n\ge6$. The case of $n=4$ was affirmatively solved later in 2011 \cite{lin}. So the only open case is when $n=5$. We reformulate the Mordell inequality conjecture as the following CPOP:
\begin{equation}\label{mordelle:eq}
\begin{cases}
\sup\limits_{\z\in\C^n} &\prod_{1\le i<j\le n}|z_i-z_j|^2\\
\,\,\rm{s.t.}&|z_1|^2+\cdots+|z_n|^2=n,\\
\end{cases}
\end{equation}
whose optimum is conjectured to be $n^n$ ($n\le5$).

The maximum of $\Delta$ is achieved only if $z\coloneqq\frac{1}{n}(z_1+\cdots+z_n)=0$. To see this, note
\begin{equation}\label{mordelle:eq1}
    \sum_{i=1}^n|z_i-z|^2=\sum_{i=1}^n|z_i|^2-n|z|^2.
\end{equation}
If $z\ne0$, then by substituting $y_i=\lambda(z_i-z)$ for $z_i$ with $\lambda=\sqrt{\frac{\sum_{i=1}^n|z_i|^2}{\sum_{i=1}^n|z_i-z|^2}}$, it holds $|y_1|^2+\cdots+|y_n|^2=n$ and the value of $\Delta$ gets larger. Therefore we can eliminate $z_n$ in \eqref{mordelle:eq} using $z_1+\cdots+z_n=0$ without changing its optimum:
\begin{equation}\label{mordelle:eq2}
\begin{cases}
\sup\limits_{\z\in\C^{n-1}} &\prod_{1\le i<j\le n-1}|z_i-z_j|^2\prod_{i=1}^{n-1}|z_i+z_1+\ldots+z_{n-1}|^2\\
\,\,\,\,\,\rm{s.t.}&|z_1|^2+\cdots+|z_{n-1}|^2+|z_1+\ldots+z_{n-1}|^2=n.\\
\end{cases}
\end{equation}

We approach \eqref{mordelle:eq2} with R-HSOS and C-HSOS. The related results are shown in Table \ref{mordelle:tab1} ($n=3$) and Table \ref{mordelle:tab2} ($n=4$), from which we see that the upper bounds provided by the two hierarchies get close to the desired value $27$ or $256$ as the relaxation order $r$ grows while R-HSOS is more efficient than C-HSOS. In addition, the local solver and R-SOS with minimum relaxation order return $27$ and $256$ for $n=3,4$, respectively.

\begin{table}[htbp]
	\caption{Results for the Mordell inequality conjecture with $n=3$.}\label{mordelle:tab1}
	\renewcommand\arraystretch{1.2}
	\centering
	\begin{tabular}{c|c|c|c|c}
		\multirow{2}{*}{$r$}&\multicolumn{2}{c|}{R-HSOS}&\multicolumn{2}{c}{C-HSOS}\\
		\cline{2-5}
	     &opt&time&opt&time\\
		\hline
		8&27.658&0.02&27.658&0.04\\
		\hline
		10&27.348&0.03&27.348&0.06\\
		\hline
		12&27.228&0.05&27.228&0.11\\
		\hline
		14&27.144&0.11&27.144&0.26\\
		\hline
		16&27.104&0.18&27.104&0.49\\
		\hline
		18&27.080&0.19&27.079&0.81\\
		\hline
		20&27.074&0.28&27.064&1.07\\
		\hline
		22&27.059&0.44&27.055&1.35\\
	\end{tabular}
\end{table}

\begin{table}[htbp]
	\caption{Results for the Mordell inequality conjecture with $n=4$.}\label{mordelle:tab2}
	\renewcommand\arraystretch{1.2}
	\centering
	\begin{tabular}{c|c|c|c|c}
		\multirow{2}{*}{$r$}&\multicolumn{2}{c|}{R-HSOS}&\multicolumn{2}{c}{C-HSOS}\\
		\cline{2-5}
	     &opt&time&opt&time\\
		\hline
		8&497.37&1.38&497.37&5.41\\
		\hline
		10&343.67&8.93&343.67&52.5\\
		\hline
		12&326.85&44.9&326.85&279\\
		\hline
		14&292.89&215&292.87&643\\
		\hline
		16&277.65&792&277.55&1979\\
            \hline
		18&274.10&1685&274.16&7048\\
            \hline
		20&269.07&4245&-&-\\
	\end{tabular}
\end{table}


\section{Application in polyphase code design}\label{app}
In this section, we provide an application of the proposed real moment-HSOS hierarchy to polyphase code design, which is a fundamental problem in signal processing \cite{deng2004polyphase}. A polyphase code set consists of $n$ signals represented by $n$ complex numbers $z_1,\ldots,z_n$ of unit norm. Given a polyphase code set $z_1,\ldots,z_n$, the related aperiodic autocorrelation functions are defined by
\begin{equation}
    A(j)=\sum_{i=1}^{n-j}z_i\overline{z}_{i+j},\quad j=1,\ldots,n-2.
\end{equation}
In polyphase code design, the goal is to minimize the norms of aperiodic autocorrelation functions $A(j),j=1,\ldots,n-2$ (called sidelobes). If we take the total autocorrelation sidelobe energy as the cost function, then the polyphase code design problem can be formulized as the following CPOP:
\begin{equation}\label{psd:eq1}
\begin{cases}
\inf\limits_{\z\in\C^n} &\sum_{j=1}^{n-2}\left|\sum_{i=1}^{n-j}z_i\overline{z}_{i+j}\right|^2\\
\,\,\,\rm{s.t.}&|z_i|^2=1,\quad i=1,\ldots,n.
\end{cases}
\end{equation}

We approach \eqref{psd:eq1} for $n=4,5,\ldots,12$ with the 5th order relaxations of the three hierarchies. The related results are shown in Table \ref{psd:tab1} where the upper bounds computed by {\tt Ipopt} are also displayed for comparison. We make the following observations: (1) for $n=4,5,6$, R-HSOS, C-HSOS, and R-SOS provide the same lower bound matching the corresponding upper bound; (2) for $n\ge7$, R-HSOS and C-HSOS provide the same lower bound matching the corresponding upper bound up to $n=10$ while R-SOS runs out of memory; (3) R-HSOS is the fastest.

\begin{table}[htbp]
	\caption{Results for the polyphase code design problem \eqref{psd:eq1} ($r=5$).}\label{psd:tab1}
	\renewcommand\arraystretch{1.2}
	\centering
	\begin{tabular}{c|c|c|c|c|c|c|c}
		\multirow{2}{*}{$n$}&\multirow{2}{*}{ub}&\multicolumn{2}{c|}{R-HSOS}&\multicolumn{2}{c|}{C-HSOS}&\multicolumn{2}{c}{R-SOS}\\
            \cline{3-8}
		&&opt&time&opt&time&opt&time\\
		\hline
		4&0.5000&0.5000&0.01&0.5000&0.02&0.5000&2.04\\
		\hline
		5&1.0000&1.0000&0.02&1.0000&0.04&1.0000&22.6\\
		\hline
		6&4.0000&4.0000&0.06&4.0000&0.18&4.0000&497\\
		\hline
		7&1.1418&1.1418&0.21&1.1418&0.38&-&-\\
		\hline
	    8&1.7428&1.7428&0.57&1.7428&1.94&-&-\\
		\hline
		9&0.0594&0.0594&2.74&0.0594&8.95&-&-\\
		\hline
		10&3.4932&3.4932&13.4&3.4932&68.7&-&-\\
		\hline
		11&2.2579&2.2096&80.4&2.2096&401&-&-\\
		\hline
		12&5.7178&-0.7769&350&-&-&-&-\\
	\end{tabular}
\end{table}

Alternatively, if we take the autocorrelation sidelobe peak as the cost function, then the polyphase code design problem can be formulized as the following optimization problem:
\begin{equation}\label{psd:eq2}
\begin{cases}
\inf\limits_{\z\in\C^n} &\max\,\left\{\left|\sum_{i=1}^{n-j}z_i\overline{z}_{i+j}\right|:j=1,\ldots,n-2\right\}\\
\,\,\,\rm{s.t.}&|z_i|^2=1,\quad i=1,\ldots,n.
\end{cases}
\end{equation}
By introducing an auxiliary variable $u$, we can reformulate \eqref{psd:eq2} as the following CPOP:
\begin{equation}\label{psd:eq3}
\begin{cases}
\inf\limits_{(\z,u)\in\C^{n+1}} &|u|\\
\,\,\,\quad\rm{s.t.}&\left|\sum_{i=1}^{n-j}z_i\overline{z}_{i+j}\right|\le|u|,\quad j=1,\ldots,n-2,\\
&|z_i|^2=1,\quad i=1,\ldots,n.
\end{cases}
\end{equation}

We approach \eqref{psd:eq3} for $n=4,5,\ldots,15$ with the 4th order relaxations of the three hierarchies. The related results are shown in Table \ref{psd:tab2} where the upper bounds computed by {\tt Ipopt} are also displayed for comparison. We make the following observations: (1) for $n=4,5,6,7$, R-HSOS, C-HSOS, and R-SOS provide the same lower bound matching the corresponding upper bound; (2) for $n\ge8$, R-HSOS and C-HSOS provide the same lower bound matching the corresponding upper bound up to $n=9$ while R-SOS runs out of memory; (3) R-HSOS is the fastest.

\begin{table}[htbp]
	\caption{Results for the polyphase code design problem \eqref{psd:eq3} ($r=4$).}\label{psd:tab2}
	\renewcommand\arraystretch{1.2}
	\centering
	\begin{tabular}{c|c|c|c|c|c|c|c}
		\multirow{2}{*}{$n$}&\multirow{2}{*}{ub}&\multicolumn{2}{c|}{R-HSOS}&\multicolumn{2}{c|}{C-HSOS}&\multicolumn{2}{c}{R-SOS}\\
            \cline{3-8}
		&&opt&time&opt&time&opt&time\\
		\hline
		4&0.5000&0.5000&0.01&0.5000&0.02&0.5000&1.43\\
		\hline
		5&0.7703&0.7703&0.02&0.7703&0.04&0.7703&10.5\\
		\hline
		6&1.0000&1.0000&0.05&1.0000&0.12&1.0000&73.2\\
		\hline
		7&0.5219&0.5219&0.14&0.5219&0.33&0.5219&732\\
		\hline
	    8&0.6483&0.6483&0.37&0.6483&0.77&-&-\\
		\hline
		9&0.1119&0.1119&1.03&0.1119&2.67&-&-\\
		\hline
		10&0.8248&0.5805&2.30&0.5805&7.21&-&-\\ 
		\hline
		11&0.6671&0.4943&6.37&0.4943&24.2&-&-\\ 
		\hline
		12&0.8971&0.4928&19.3&0.4928&87.8&-&-\\ 
		\hline
		13&0.7154&0.4189&105&0.4189&440&-&-\\ 
		\hline
		14&0.8146&0.3309&324&0.3309&1548&-&-\\ 
		\hline
		15&0.9232&0.3098&1096&0.3098&5546&-&-\\ 
	\end{tabular}
\end{table}

\section{Conclusions}
In this paper, we have presented a real moment-HSOS hierarchy for CPOPs with real coefficients, which offers a new item in the arsenal of structure-exploiting techniques for moment-HSOS hierarchies of CPOPs. One possible topic for future research is developing a structured moment-HSOS hierarchy of CPOPs for which the polynomial data are invariant under conjugation of subsets of variables.
Besides, inspired by the complexity comparison presented in Table \ref{tab:comp}, we are wondering whether it is advantageous to solve certain real POPs via the complex moment-HSOS hierarchy after converting them to equivalent complex POPs? We give an illustrative example below, and leave a thorough investigation to future research.

\begin{example}
Consider the following real POP:
\begin{equation}\label{ex:rpop}
\begin{cases}
\inf\limits_{x_1,x_2,x_3,x_4\in\R}& 3-x_1^2-x_3^2+x_1x_2^2+2x_2x_3x_4-x_1x_4^2\\
\,\,\,\,\,\,\quad\rm{s.t.}&x_2\ge0,\\
&x_1^2+3x_3^2-2=0,\\
&x_4=0,\\
&x_1^2+x_2^2+x_3^2+x_4^2-3=0.
\end{cases}
\end{equation}
The second-order (moment-SOS) relaxation gives a lower bound $-0.414213$ which is globally optimal as the moment matrix is of rank one.

Note that \eqref{ex:rpop} can be equivalently expressed as the following CPOP by letting $z_1=x_1+x_3\i,z_2=x_2+x_4\i$:
\begin{equation}
\begin{cases}
\inf\limits_{z_1,z_2\in\C}& 3-|z_1|^2+\frac{1}{2}z_1\overline{z}_2^2+\frac{1}{2}z_2^2\overline{z}_1\\
\,\,\,\,\,\rm{s.t.}&z_2+\overline{z}_2\ge0,\\
&|z_1|^2-\frac{1}{4}z_1^2-\frac{1}{4}\overline{z}_1^2-1=0,\\
&z_2^2+\overline{z}_2^2-2|z_2|^2=0,\\
&|z_1|^2+|z_2|^2-3=0.
\end{cases}
\end{equation}
The second-order (real moment-HSOS) relaxation gives a lower bound $-0.909535$, and the third-order relaxation achieves the global optimum $-0.414213$. Despite the real moment-HSOS hierarchy needs a higher order relaxation to attain global optimality, we note that the third-order complex moment matrix is of size $10$ while the second-order real moment matrix is of size $15$, which means for this specific problem, the complex reformulation is more preferable from the point of view of computational complexity.
\end{example}

\bibliographystyle{siamplain}
\bibliography{refer}
\end{document}